\documentclass[12pt]{amsart}
\usepackage{color,epsfig}

\usepackage[dvipsnames]{xcolor}

\usepackage[left=3.4cm,right=3.4cm,top=3.5cm,bottom=3cm]{geometry}

\DeclareMathAlphabet{\mathpzc}{OT1}{pzc}{m}{it}

\usepackage[all]{xy}
\usepackage{latexsym}
\usepackage{graphicx}
\usepackage{tikz}
\usepackage{transparent}
\usepackage{graphics}
\usepackage{hyperref}

\usepackage[backend=bibtex]{biblatex}
\newcommand{\zehn}{\hspace{10pt}}
\newcommand{\fuenf}{\hspace{5pt}}
\newcommand{\fuenfm}{\hspace{-5pt}}

\numberwithin{figure}{section}
\numberwithin{table}{section}
\newtheorem{theorem}{Theorem}[section]

\theoremstyle{definition}
\newtheorem{definition}[theorem]{Definition}

\newtheorem{question}[theorem]{Question}

\theoremstyle{remark}
\newtheorem{remark}[theorem]{Remark}

\numberwithin{equation}{section}
\newfont{\tap}{tap scaled 650}

\def \[{[ }
\def \]{] }

\definecolor{dgreen}{rgb}{0,0.5,0}        
\definecolor{dred}{rgb}{0.5,0,0}

\begin{document}

\title{Friezes for a pair of pants}

\author{\.{I}lke \c{C}anak\c{c}i, Anna Felikson, Ana Garcia Elsener and Pavel~Tumarkin}

\address{\.{I}lke \c{C}anak\c{c}i, Department of Mathematics, VU Amsterdam, De Boelelaan 1105, 1081 HV Amsterdam, The Netherlands}
\email{i.canakci@vu.nl}

\address{Anna Felikson and Pavel Tumarkin, Department of Mathematical Sciences, Durham University, Mathematical Sciences \& Computer Science Building,
Upper Mountjoy Campus,
Stockton Road, Durham, DH1 3LE, UK}
\email{anna.felikson@durham.ac.uk, pavel.tumarkin@durham.ac.uk}

\address{Ana Garcia Elsener, School of Mathematics and Statistics, University of Glasgow, University Place,
Glasgow G12 8SQ, UK}
\email{anaclara.garciaelsener@glasgow.ac.uk}



\begin{abstract}
Frieze patterns are numerical arrangements that satisfy a local arithmetic rule. These arrangements are actively studied in connection to the theory of cluster algebras. In the setting of  cluster algebras, the notion of a frieze pattern can be generalized, in particular to a frieze  associated with a bordered marked surface endowed with a decorated hyperbolic metric. We study friezes associated with a pair of pants, interpreting entries of the frieze as $\lambda$-lengths of arcs connecting the marked points. We prove that all positive integral friezes over such surfaces are unitary, i.e. they arise from triangulations with all edges having unit $\lambda$-lengths.
\end{abstract}

\thanks{AGE was partially supported by EPSRC grant 'The Homological Minimal Model Program' EP/R009325/1, PT was partially supported by the Leverhulme Trust research grant RPG-2019-153}

\maketitle

\section{Friezes from marked surfaces}\label{section-friezes-on-surfaces}

Frieze patterns were introduced by Coxeter in~\cite{Co}, and studied by Conway\,--\:Coxeter in~\cite{CoCo,CoCo2}. In these works the authors connect the notion of frieze patterns with formulas for the {\it pentagramma mirificum}, continued fractions and triangulations of polygons. Coxeter's invention reinvigorated in the early 2000's with the appearance of triangulations of polygons in relation to the cluster algebras introduced by Fomin\,--\:Zelevinsky~\cite{FZ}. 


\begin{definition}
A \emph{frieze pattern (or simply frieze) of type $A_{n}$} (or Conway\,--\:Coxeter frieze) is a grid of integers consisting of $n+2$ infinite rows where successive rows are displayed with a shift and where the first and the last rows consist entirely of 1's and the remaining entries are positive integers.  In addition, the entries in the grid satisfy the \emph{diamond rule} which asserts that every diamond formed by the neighboring entries of the form $\begin{smallmatrix}&a\\b&&c\\&d\end{smallmatrix}$ satisfies the identity $bc-ad=1$. The first nontrivial row is called the \emph{quiddity row} of the frieze.
\end{definition}

 \begin{figure}[h]
\centering
\scalebox{0.8}[0.8]{\def\svgwidth{5.5in}
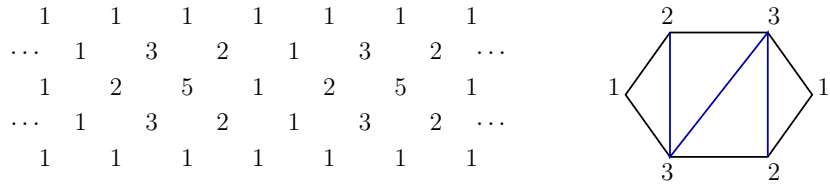}
\caption{Conway\,--\:Coxeter frieze and  triangulation on the hexagon that defines the frieze of type $A_3$.}
\label{f-polygon}
\end{figure}

It was shown by Coxeter that type~$A_n$ frieze patterns are periodic with period dividing $n+3$ and they have a glide symmetry; hence we may consider their fundamental domain and simply refer to them as friezes. Furthermore, it was proven by Conway and Coxeter that type $A_n$ friezes are in bijection with triangulations of convex $(n+3)$-gons. Given such a triangulation (see Fig.~\ref{f-polygon} for an example), consider the sequence of vertices counterclockwise and count the number of triangles incident to each vertex. This gives the \emph{quiddity sequence}; that is, the quiddity row modulo the period of the frieze which determines the type~$A_n$ frieze by applying the diamond rule recursively.

\smallskip

Inspired by Conway and Coxeter's approach using triangulations of the polygon, one can consider friezes on marked surfaces with a hyperbolic metric, following~\cite{FST,Pe} in relation to cluster algebras associated with such surfaces.

Let $S$ denote an oriented Riemann surface, with a non-empty boundary that we denote $\partial S$, and let $M  \subset \partial S$ be a finite subset. Furthermore, we require that each boundary component has at least one point in $M$. The elements of $M$ are called \emph{marked points}, and the pair $(S,M)$ is a \emph{marked surface}.

An \emph{arc} in $(S,M)$ is a curve $\gamma$ in $S$ such that its endpoints are marked points, and it is disjoint from $M$ and $\partial S$ otherwise. An arc is considered up to
isotopy relative to its endpoints. We require that arcs do not self-cross, except possibly at the endpoints and that they are not contractible.  
For each arc 
(i.e. each isotopy class), we will consider the geodesic representative.

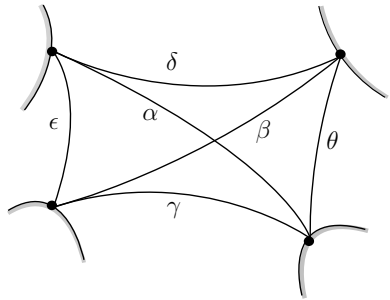
\begin{figure}[h]
\centering
\scalebox{0.8}[0.8]{\def\svgwidth{2.5in}
\begingroup%
  \makeatletter%
  \providecommand\color[2][]{%
    \errmessage{(Inkscape) Color is used for the text in Inkscape, but the package 'color.sty' is not loaded}%
    \renewcommand\color[2][]{}%
  }%
  \providecommand\transparent[1]{%
    \errmessage{(Inkscape) Transparency is used (non-zero) for the text in Inkscape, but the package 'transparent.sty' is not loaded}%
    \renewcommand\transparent[1]{}%
  }%
  \providecommand\rotatebox[2]{#2}%
  \newcommand*\fsize{\dimexpr\f@size pt\relax}%
  \newcommand*\lineheight[1]{\fontsize{\fsize}{#1\fsize}\selectfont}%
  \ifx\svgwidth\undefined%
    \setlength{\unitlength}{208.00604153bp}%
    \ifx\svgscale\undefined%
      \relax%
    \else%
      \setlength{\unitlength}{\unitlength * \real{\svgscale}}%
    \fi%
  \else%
    \setlength{\unitlength}{\svgwidth}%
  \fi%
  \global\let\svgwidth\undefined%
  \global\let\svgscale\undefined%
  \makeatother%
  \begin{picture}(1,0.77257777)%
    \lineheight{1}%
    \setlength\tabcolsep{0pt}%
    \put(0,0){\includegraphics[width=\unitlength,page=1]{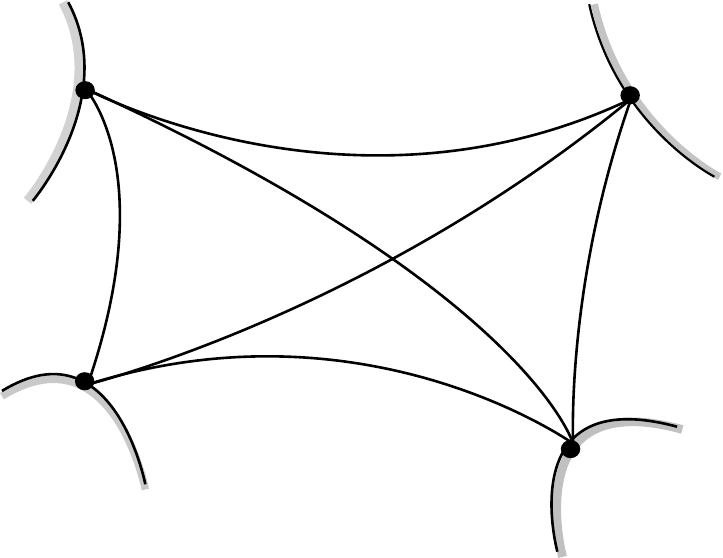}}%
    \put(0.35471195,0.46796702){\color[rgb]{0,0,0}\makebox(0,0)[lt]{\lineheight{1.25}\smash{\begin{tabular}[t]{l}$\alpha$\end{tabular}}}}%
    \put(0.65141051,0.4207639){\color[rgb]{0,0,0}\makebox(0,0)[lt]{\lineheight{1.25}\smash{\begin{tabular}[t]{l}$\beta$\end{tabular}}}}%
    \put(0.41663052,0.59908879){\color[rgb]{0,0,0}\makebox(0,0)[lt]{\lineheight{1.25}\smash{\begin{tabular}[t]{l}$\delta$\end{tabular}}}}%
    \put(0.41845158,0.22393495){\color[rgb]{0,0,0}\makebox(0,0)[lt]{\lineheight{1.25}\smash{\begin{tabular}[t]{l}$\gamma$\end{tabular}}}}%
    \put(0.10885862,0.44429226){\color[rgb]{0,0,0}\makebox(0,0)[lt]{\lineheight{1.25}\smash{\begin{tabular}[t]{l}$\epsilon$\end{tabular}}}}%
    \put(0.83549156,0.39330044){\color[rgb]{0,0,0}\makebox(0,0)[lt]{\lineheight{1.25}\smash{\begin{tabular}[t]{l}$\theta$\end{tabular}}}}%
  \end{picture}%
\endgroup%
}
\caption{Ptolemy relation: $\lambda_\alpha \lambda_\beta = \lambda_\gamma \lambda_\delta + \lambda_\epsilon \lambda_\theta$.}
\label{fig:ptolemy}
\end{figure}

Two arcs $\gamma_1$ and $\gamma_2$ in $S$ are \emph{non-crossing} if there exist curves in their relative isotopy classes that are non-intersecting except possibly at their endpoints. A \emph{triangulation} $T$ of $(S,M)$ is defined as a maximal collection of pairwise non-crossing arcs.

The marked surface $(S,M)$ can be endowed with a hyperbolic metric having a cusp at each marked point.
After choosing a horocycle at every  marked point, 
each arc $\gamma$ can be assigned a finite number called {\it $\lambda$-length} and denoted by $\lambda_\gamma$, see~\cite{Pe} (it is defined as $\lambda_\gamma=e^{l/2}$, where $l$ is the signed hyperbolic distance between the horocycles centred at the endpoints of $\gamma$). 


The most important property that will be used throughout the article is the following.

\begin{remark}\label{remark-ptolemy}
For each quadrilateral as in Fig. \ref{fig:ptolemy} the $\lambda$-lengths  satisfy the \emph{Ptolemy relation.}
\end{remark}


Entries in a Conway\,--\:Coxeter frieze can be interpreted as $\lambda$-lengths on the disk in such a way that the 1's on the trivial rows correspond to boundary edge $\lambda$-lengths. In this interpretation, the diamond rule is a particular case of the Ptolemy relation. More explicitly, with the notation of Fig.~\ref{fig:ptolemy}, if the arcs $\epsilon$ and $\theta$ or $\delta$ and $\gamma$ are boundary edges then we recover the diamond rule.

\begin{question}\label{question-unitary} Let $(S,M)$ be a marked surface with a hyperbolic structure having a cusp at each marked point.
Suppose that we associate  a  horocycle to every marked point in such a way that for each arc the $\lambda$-length is a positive integer, and all boundary segments have unit $\lambda$-lengths. Does there exist a triangulation $T$ of the surface in which $\lambda_\alpha = 1$ for each $\alpha \in T$?
\end{question}

\begin{remark}
  \label{poly+annulus}
  Conway and Coxeter~\cite{CoCo,CoCo2} give an affirmative answer to Question~\ref{question-unitary} for the case when $(S,M)$ is a disk with boundary marked points.
 Also,   it was proven by Gunawan and Schiffler~\cite{GuS} that the answer is positive when $S$ is an annulus  with boundary marked points.

\end{remark}  


In Section~\ref{section-cluster-friezes}, we explain how  Question~\ref{question-unitary} is related to recent developments in cluster theory.  In Section~\ref{section-main-result}, we show that the answer to Question~\ref{question-unitary} is affirmative for a family of marked surfaces arising from a pair of pants.


\section{Friezes from cluster theory}\label{section-cluster-friezes}

In this section we discuss the relation of  Question~\ref{question-unitary} to the theory of cluster algebras. 
One can find a good overview on friezes and their interplay with cluster theory written by Morier--Genoud~\cite{Mo}, and a nice introduction to cluster algebras and cluster algebras of surface type in~\cite[Chapter 3]{Ass}. 

\smallskip

Let $Q$ be a cluster quiver, i.e. $Q$ is without loops and $2$-cycles. We may consider frieze patterns as homomorphisms from the cluster algebra associated with $Q$ to the integers as appeared in~\cite{GuS} and~\cite{FP}.

\begin{definition} Let $Q$ be a cluster quiver and $\mathcal{A}(Q)$ is the associated cluster algebra with trivial coefficients.
\begin{enumerate}
\item A \emph{frieze} or \emph{cluster frieze} associated to $\mathcal{A}(Q)$ is a ring homomorphism $\lambda: \mathcal{A}(Q) \to \mathbb Z$.
\item A frieze $\lambda$ is \emph{positive} if for any cluster variable $x \in \mathcal{A}(Q)$, its image $\lambda(x)$ is in $\mathbb Z_+$. 
\item A positive 
  frieze $\lambda$ is \emph{unitary} if there exists a cluster $X$ in $\mathcal{A}(Q)$ such that every cluster variable $x_i \in X$ is mapped to $1$ by $\lambda $.
\end{enumerate}
\end{definition}

Consider the  repetition quiver of Dynkin type $A$, $D$ or $E$, or affine type (see~\cite[Section 2.1]{Mo}). Set positive entries in this quiver in such a way that the diamond rule now is now adapted to the mesh rules. In Fig.~\ref{fig:frieze-cluster-category} we can see  an example of cluster frieze of type $A_3$ by the evaluation $x_i \to 1$ for each $i$. 

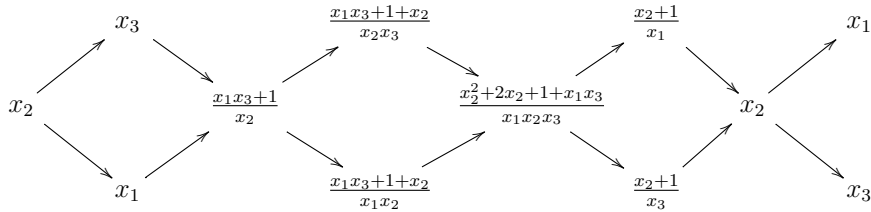
\begin{figure}[ht]
$$\scalebox{0.85}{
\xymatrix@R-10pt@C-10pt{ 
&\zehn x_3\zehn  \ar[rd] &&
 \frac{x_1x_3+1+x_2}{x_2x_3} \ar[rd] &&
 \frac{x_2+1}{x_1} \ar[rd] &&
 \zehn x_1
 \zehn \\
\fuenf x_2\fuenf \ar[rd]\ar[ru] &&
 \frac{x_1x_3+1}{x_2} \ar[rd]\ar[ru] &&
\fuenfm \fuenfm \frac{x_2^2+2x_2+1+x_1x_3}{x_1x_2x_3}\fuenfm\fuenfm \ar[rd]\ar[ru] &&
 \fuenf x_2\fuenf  \ar[rd]\ar[ru] &&
\\
& x_1 \ar[ru] &&
  \frac{x_1x_3+1+x_2}{x_1x_2} \ar[ru] &&
  \frac{x_2+1}{x_3} \ar[ru] &&
 x_3 \\
} 
}
$$
\caption{Cluster variables displayed in  the repetition quiver of type $A_3$. The frieze in Fig. \ref{f-polygon} is obtained via evaluation $x_i \to 1$, and adding the first and last rows with all 1's.}
\label{fig:frieze-cluster-category}
\end{figure}

The mesh rule in type $D$ is interpreted as follows: whenever there is a configuration $\begin{smallmatrix}&a\\b&&c\\d&&e\end{smallmatrix}$ of neighboring entries in the mesh quiver then $bc-a=de-a=1$, and whenever there is $\begin{smallmatrix}&a\\b&&c\\&d\\&e\end{smallmatrix}$ then $bc-ade=1$, see Fig.~\ref{fig:friezeD}.

\begin{figure}[h!]
$$\scalebox{.85}{\xymatrix@C=-2pt@R=-2pt{
1&&1&&1&& 1\\
& 2&& 2 && 2 && 2 \\
3&& 3 && 3 && 3 \\
 & 2 && 2 && 2 && 2 \\
 &2 && 2 && 2 && 2 }}$$
\caption{Frieze of type $D_4$.}
\label{fig:friezeD}
\end{figure}
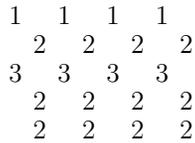

\begin{remark}\label{remark-friezes-cluster}
The results of Conway and Coxeter~\cite{CoCo2} and of Gunawan and Schiffler~\cite[Theorem 4.2]{GuS} mentioned in Remark~\ref{poly+annulus} can be now reformulated as follows: {\it all positive friezes of type $A$ and $\widetilde A$, respectively, are unitary}.
\end{remark}

\begin{remark}\label{remark-non-unitary}
Not all friezes of type $D$, $E$, $\widetilde{D}$ and $\widetilde{E}$ are unitary, see~\cite{FP} and \cite[Section 4.2.1]{GuS}. For instance, the frieze in Fig.~\ref{fig:friezeD} is an example of a non-unitary frieze of type $D$.
\end{remark}

Fomin, Shapiro and Thurston defined cluster algebras of surface type in~\cite{FST}. Fomin and Thurston continued this development in~\cite{FT}.  In these articles, the authors find a correspondence between  $\lambda$-lengths of arcs and cluster variables, so that  a triangulation represents a cluster and mutations of the cluster can be interpreted as flips of arcs in the triangulation of the surface~\cite[Proposition 7.6]{FT}.  
Fixing a triangulation $T$ of $(S,M)$, the collection of $\lambda$-lengths corresponding to the
arcs in the triangulation (including those for boundary segments) forms a system of coordinates
for the decorated Teichm\"uller space. Choosing another triangulation results in a different coordinate chart, but
all the triangulations for $(S,M)$ are  related by sequences of flips, and the cluster variables in the adjacent clusters are related by Ptolemy relations.

\smallskip


We can think that a (positive) frieze $\lambda$ over a cluster algebra of surface type is defined by assigning  to each cluster variable, i.e. each arc, a $\lambda$-length in $\mathbb{Z}_+$, see Definition~\ref{fr}.

The question of a frieze being unitary has been studied  for cluster algebras of acyclic types $A$, $D$, $E$, $\widetilde{A}$, $\widetilde{D}$ and $\widetilde{E}$, see Remarks~\ref{remark-friezes-cluster}~and~\ref{remark-non-unitary}.  However, most cluster algebras of surface type are non-acyclic. In the next section we propose a new approach to answer Question~\ref{question-unitary} in a non-acyclic setting.


\section{Friezes for a pair of pants}\label{section-main-result}

Let $S$ be a topological surface known as \emph{a pair of pants}, i.e., $S$ is a sphere with three open disks removed.
Let $M$ be a set of marked points on the boundary of $S$ such that each boundary component contains at least one marked point. Consider the marked surface $(S,M)$. 

We have defined (cluster) friezes in terms of corresponding cluster algebras. Restricting the definition to the case of cluster algebras from surfaces and using results of~\cite{FST} and~\cite{FT} we obtain the following reformulation.

\begin{definition}
\label{fr}
  By a \emph{positive frieze} $\lambda(S,M)$ from the marked surface $(S,M)$ we will understand a map
 $\lambda: \gamma \to \lambda_\gamma$  from the set of arcs on $S$ to positive integers  $\mathbb{Z}_+$  such that the Ptolemy relation is satisfied for each quadrilateral on $S$. 
\end{definition}

We denote the map above $\lambda$ (and  call $\lambda_\gamma$ a {\it $\lambda$-length} of $\gamma$) as every such map can be understood as taking $\lambda$-lengths of arcs for a decorated hyperbolic structure on $(S,M)$. At the same time, one can understand this map formally, without remembering about the hyperbolic structure.

The following notions will be used in the statement and the proof of our main result. 

\begin{definition} Let $(S,M)$ be a marked surface and $\lambda=\lambda(S,M)$ be a frieze.
  \begin{enumerate}
  \setlength\itemsep{0.01em} 

\item A triangulation $T$ of $(S,M)$ is {\it unitary} for $\lambda$ if  $\lambda_\gamma=1$ for each $\gamma\in T$.   
\item The frieze  $\lambda$ is \emph{unitary} if there exists a unitary triangulation $T$ of $(S,M)$.
   
\item A non-boundary arc $\gamma$ on $S$ is \emph{peripheral} if it can be isotopically deformed to the boundary of $S$, otherwise $\gamma$ is \emph{bridging}. 
\item  A \emph{short diagonal} is a peripheral arc that can be isotopically deformed to a union of two consecutive  boundary segments.  
\item A surface together with a frieze  $\lambda$ is \emph{reduced} if there is no peripheral arc $\gamma\in (S,M)$ such that $\lambda_\gamma=1$.
\end{enumerate}
\end{definition}

\begin{theorem}
  Let $(S,M)$ be a pair of pants with a set of marked points $M$,  and let $\lambda(S,M)$ be a positive frieze such that $\lambda_\delta=1$ for every boundary arc $\delta$. Then
  \begin{itemize}
  \item[(a)] the frieze $\lambda(S,M)$  is unitary;
\item[(b)]the unitary triangulation $T$ of $(S,M)$  is unique; 
\item[(c)] for any short diagonal $\alpha$ in  $(S,M)$,  its image $\lambda_\alpha$ is equal to the number of triangles of $T$ crossed by $\alpha$, where $T$ is the unitary triangulation from part (b). 
\end{itemize}  
  \end{theorem}

  \begin{proof}

  First, we will prove part (a).

\medskip

\noindent
{\bf Reducing $S$.} First, we reduce $S$ as follows: if $S$ has a peripheral arc $\psi$ such that $\lambda_\psi=1$, then we cut $S$ along $\psi$ and separate a polygon $P_{\psi}$ from $S$ (in this polygon all boundary arcs are of unit $\lambda$-length). After cutting $P_{\psi}$ from $S$, the number of boundary marked points will decrease. Then we repeat the process: if there is a peripheral arc  of $\lambda$-length 1 in the resulting surface $S\setminus P_\psi$, we cut along it. If the number of marked points on some boundary component equals 1 then there are no peripheral arcs on that boundary anymore. So, after cutting finitely many times we will get a reduced surface $\widetilde{S}$.
Notice that if we know that $\widetilde{S}$ has a unitary triangulation then we can obtain a unitary triangulation for $S$ by applying the Conway\,--\:Coxeter result to each of the polygons cut out in the process of reducing $S$. So, it is sufficient to show the statement for the reduced surface $\widetilde{S}$. 
From now on we will assume that $S$ is reduced.

  \medskip
  \noindent
  {\bf Cutting $S$.}
  Let $\gamma$ be a shortest bridging arc whose endpoints don't lie on the same boundary component on $S$, i.e. $\lambda_\gamma=k$ where $k$ is minimal over all such bridging arcs $\gamma$. We will cut $S$ along $\gamma$ and denote by $S'$ the obtained annulus, see  Fig.~\ref{f-2}.

  Suppose that $k=1$. Then $S'$ together with the restriction of $\lambda$ to $S'$ is a positive frieze from an annulus such that $\lambda_\delta=1$ for each boundary arc $\delta$.
  Every such frieze is unitary~\cite{GuS}. Gluing the surface back along $\gamma$ we obtain a unitary triangulation on $S$.

\begin{figure}[h]
\centering
\scalebox{0.7}[0.7]{\def\svgwidth{6in}
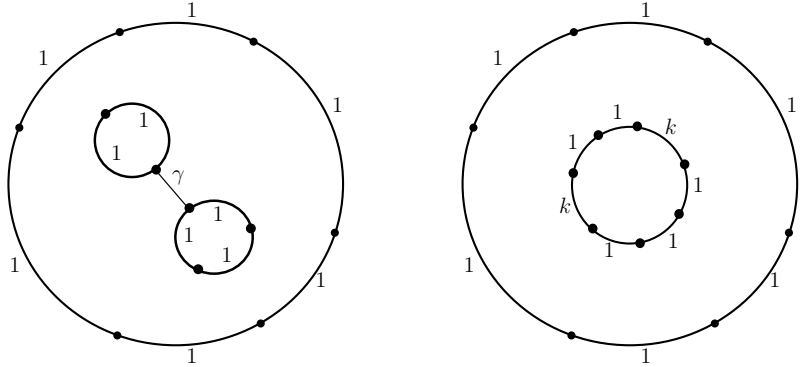}
\caption{After cutting $S$ along the arc $\gamma$ with $\lambda_{\gamma}=k$ we obtain the surface $S'$ (right).}
\label{f-2}
\end{figure}

  From now on we assume $k>1$, i.e. no bridging arc on $S'$ is of unit $\lambda$-length. Our aim now is to obtain a contradiction.

  \medskip
  \noindent
  {\bf Constructing a triangulation on $S'$.}
  Choose $\gamma_0$ to be one of the shortest bridging arcs on $S'$ (whose endpoints necessarily lie in the two different boundaries of $S'$ since $S'$ is an annulus). Next, choose $\gamma_1$ to be a shortest bridging arc on $S'$ not intersecting $\gamma_0$. Choose inductively $\gamma_i$ to be a shortest bridging arc on $S'$ not intersecting any of the previously chosen arcs $\gamma_0,\gamma_1\dots,\gamma_{i-1}$. After repeating this finitely many times we will arrive to a triangulation of $S'$ consisting only of bridging arcs.
  Observe that by this construction we have an annulus with $\lambda_\delta = 1$ for each boundary arc $\delta$, except for two arcs of $\lambda$-length $k>1$ lying in the same boundary component and not following one another consecutively, see  Fig.~\ref{f-2}, right.
  
Lift the triangulation to the universal cover, so that we obtain a triangulation $T$  of an infinite strip $\Sigma$ by bridging arcs, see  Fig.~\ref{f-1a} (left).

We will label the arcs on $\Sigma$ as follows. Let $\alpha_0$ be a lift of $\gamma_0$. Let $l$ be  the line obtained as the lift of the non-contractible simple closed curve which is disjoint from the boundary (see \cite[Section 3]{Pe}). It will cross all lifts of all non-peripheral arcs.  Choose an orientation of $l$. 
We label the arc $\alpha_i$ if the crossing of $l$ with the arc is the $i$-th crossing of $l$ with the triangulation when following $l$ from $\alpha_0$ in the positive direction, see Fig.~\ref{f-1a} (left). 
Denote $a_i=\lambda_{\alpha_i}$ for $i\in \mathbb{Z}$.



  \medskip
  \noindent
  {\bf Claim 1:}
  Let $\beta$ be a short diagonal enclosing two boundary segments, one of unit $\lambda$-length and the other one of $\lambda$-length $k$, as in Fig. \ref{f-1a} (right) (notice that $\beta \notin T$). Then $\lambda_\beta > k$.

  \begin{figure}[ht]
\centering
\scalebox{0.8}[0.75]{\def\svgwidth{4.8in}
\begingroup%
  \makeatletter%
  \providecommand\color[2][]{%
    \errmessage{(Inkscape) Color is used for the text in Inkscape, but the package 'color.sty' is not loaded}%
    \renewcommand\color[2][]{}%
  }%
  \providecommand\transparent[1]{%
    \errmessage{(Inkscape) Transparency is used (non-zero) for the text in Inkscape, but the package 'transparent.sty' is not loaded}%
    \renewcommand\transparent[1]{}%
  }%
  \providecommand\rotatebox[2]{#2}%
  \newcommand*\fsize{\dimexpr\f@size pt\relax}%
  \newcommand*\lineheight[1]{\fontsize{\fsize}{#1\fsize}\selectfont}%
  \ifx\svgwidth\undefined%
    \setlength{\unitlength}{351.20191075bp}%
    \ifx\svgscale\undefined%
      \relax%
    \else%
      \setlength{\unitlength}{\unitlength * \real{\svgscale}}%
    \fi%
  \else%
    \setlength{\unitlength}{\svgwidth}%
  \fi%
  \global\let\svgwidth\undefined%
  \global\let\svgscale\undefined%
  \makeatother%
  \begin{picture}(1,0.29412865)%
    \lineheight{1}%
    \setlength\tabcolsep{0pt}%
    \put(0,0){\includegraphics[width=\unitlength,page=1]{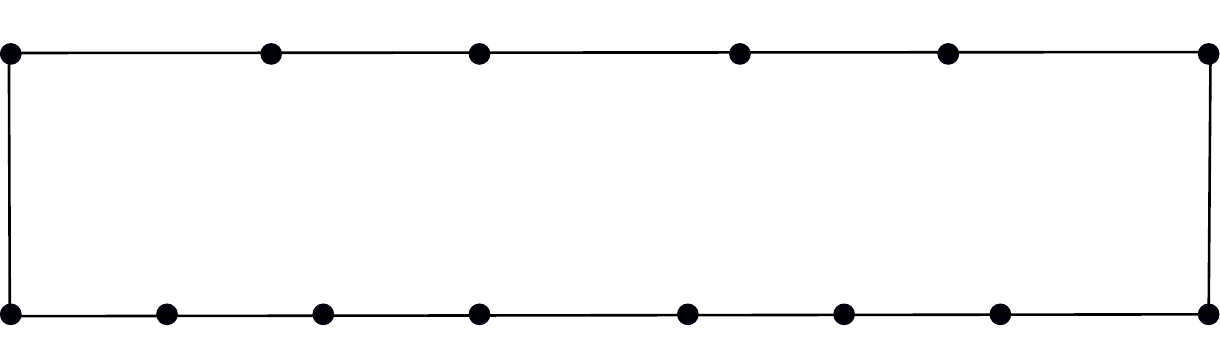}}%
    \put(0.31924324,0.00401523){\color[rgb]{0,0,0}\makebox(0,0)[lt]{\lineheight{1.25}\smash{\begin{tabular}[t]{l}$k$\end{tabular}}}}%
    \put(0.75061864,0.00401523){\color[rgb]{0,0,0}\makebox(0,0)[lt]{\lineheight{1.25}\smash{\begin{tabular}[t]{l}$k$\end{tabular}}}}%
    \put(0,0){\includegraphics[width=\unitlength,page=2]{strip.pdf}}%
    \put(0.01828701,0.09370719){\color[rgb]{0,0,0}\makebox(0,0)[lt]{\lineheight{1.25}\smash{\begin{tabular}[t]{l}$\alpha_0$\end{tabular}}}}%
    \put(0,0){\includegraphics[width=\unitlength,page=3]{strip.pdf}}%
    \put(0.06099739,0.19194134){\color[rgb]{0,0,0}\makebox(0,0)[lt]{\lineheight{1.25}\smash{\begin{tabular}[t]{l}$\alpha_1$\end{tabular}}}}%
    \put(0.17204462,0.07235199){\color[rgb]{0,0,0}\makebox(0,0)[lt]{\lineheight{1.25}\smash{\begin{tabular}[t]{l}$\alpha_2$\end{tabular}}}}%
    \put(0,0){\includegraphics[width=\unitlength,page=4]{strip.pdf}}%
    \put(0.01630397,0.27659956){\color[rgb]{0,0,0}\makebox(0,0)[lt]{\lineheight{1.25}\smash{\begin{tabular}[t]{l}$l$\end{tabular}}}}%
  \end{picture}%
\endgroup%
} \hspace{15pt} {\def\svgwidth{2.2in}
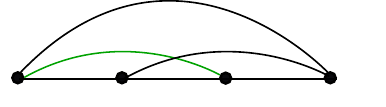} 
\caption{Left: triangulation of $\Sigma$ by bridging arcs. Right: a quadrilateral in $\Sigma$ }
\label{f-1a}
\end{figure}  

 \medskip
  \noindent
  {\it Proof of Claim 1:} Let $v_1v_2v_3v_4$ be a quadrilateral in $\Sigma$ such that all vertices are consecutive marked points on the same boundary component, so that $\beta = v_1 v_3$. We have $\lambda_{v_1v_2}=\lambda_{v_3v_4}=1$ and $\lambda_{v_2v_3}=k$ by construction. The Ptolemy relation on the quadrilateral  $v_1v_2v_3v_4$ gives
  $$\lambda_{v_1v_3}\lambda_{v_2v_4}=k\lambda_{v_1v_4}+1.
  $$
Notice that both $v_1v_3$ and $v_2v_4$ were bridging arcs in $S'$ since they cross over the boundary segment of $\lambda$-length $k$, so we have $\lambda_{v_1v_3}\ge k$. 
Suppose that $\lambda_{v_1v_3}=k$.  Then the left hand side of the equation above is divisible by $k$, while the right hand side is not, which is impossible. So, we conclude  $\lambda_{v_1v_3}>k$. \\ 
{\it End of the proof of Claim 1.}

  \medskip
  \noindent
  {\bf Claim 2:} $a_{i+1}\ge a_i$ for all $i\ge 0$.

 \medskip
  \noindent
  {\it Proof of Claim 2:}
  For $i=0$ the claim follows from the construction of the triangulation as $\alpha_0$ is the shortest bridging arc on $S'$.
  Fix some $i>0$ and suppose that the claim holds for all $i' < i$.
  
   \begin{figure}[h]
\centering
\scalebox{0.8}[0.7]{\def\svgwidth{3.7in}
\begingroup%
  \makeatletter%
  \providecommand\color[2][]{%
    \errmessage{(Inkscape) Color is used for the text in Inkscape, but the package 'color.sty' is not loaded}%
    \renewcommand\color[2][]{}%
  }%
  \providecommand\transparent[1]{%
    \errmessage{(Inkscape) Transparency is used (non-zero) for the text in Inkscape, but the package 'transparent.sty' is not loaded}%
    \renewcommand\transparent[1]{}%
  }%
  \providecommand\rotatebox[2]{#2}%
  \newcommand*\fsize{\dimexpr\f@size pt\relax}%
  \newcommand*\lineheight[1]{\fontsize{\fsize}{#1\fsize}\selectfont}%
  \ifx\svgwidth\undefined%
    \setlength{\unitlength}{338.53470314bp}%
    \ifx\svgscale\undefined%
      \relax%
    \else%
      \setlength{\unitlength}{\unitlength * \real{\svgscale}}%
    \fi%
  \else%
    \setlength{\unitlength}{\svgwidth}%
  \fi%
  \global\let\svgwidth\undefined%
  \global\let\svgscale\undefined%
  \makeatother%
  \begin{picture}(1,0.24675352)%
    \lineheight{1}%
    \setlength\tabcolsep{0pt}%
    \put(0,0){\includegraphics[width=\unitlength,page=1]{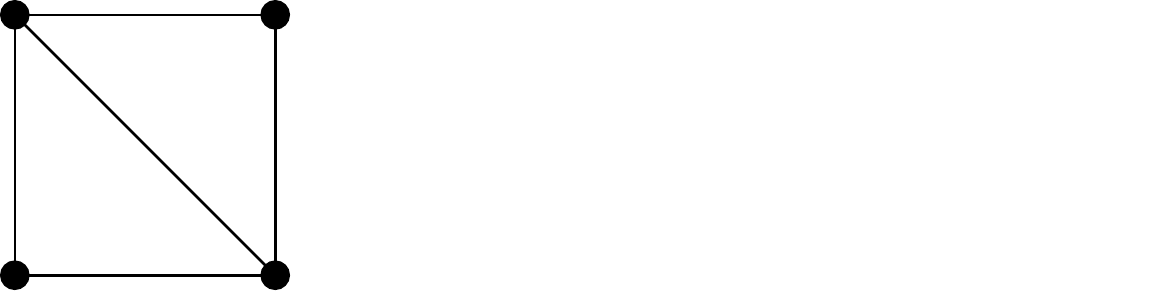}}%
    \put(0.02447369,0.07115594){\color[rgb]{0,0,0}\makebox(0,0)[lt]{\lineheight{1.25}\smash{\begin{tabular}[t]{l}$a_{i-1}$\end{tabular}}}}%
    \put(0.12337679,0.14315738){\color[rgb]{0,0,0}\makebox(0,0)[lt]{\lineheight{1.25}\smash{\begin{tabular}[t]{l}$a_i$\end{tabular}}}}%
    \put(0.24997285,0.12654166){\color[rgb]{0,0,0}\makebox(0,0)[lt]{\lineheight{1.25}\smash{\begin{tabular}[t]{l}$a_{i+1}$\end{tabular}}}}%
    \put(0,0){\includegraphics[width=\unitlength,page=2]{figure3.pdf}}%
    \put(0.55538565,0.07273839){\color[rgb]{0,0,0}\makebox(0,0)[lt]{\lineheight{1.25}\smash{\begin{tabular}[t]{l}$a_{i-1}$\end{tabular}}}}%
    \put(0.69938846,0.09441794){\color[rgb]{0,0,0}\makebox(0,0)[lt]{\lineheight{1.25}\smash{\begin{tabular}[t]{l}$a_i$\end{tabular}}}}%
    \put(0.76395276,0.13983425){\color[rgb]{0,0,0}\makebox(0,0)[lt]{\lineheight{1.25}\smash{\begin{tabular}[t]{l}$a_{i+1}$\end{tabular}}}}%
  \end{picture}%
\endgroup%
}
\caption{Quadrilaterals in $T$.}
\label{f-3}
\end{figure}

  Denote by $t_n$ the triangle of $T$ with sides $\alpha_{n},\alpha_{n+1}$ and a boundary side.  
  Consider the quadrilateral $q_i$ formed by triangles $t_{i-1}$ and $t_i$. Up to swapping top boundary of the strip with the bottom, the quadrilateral $q_i$ looks like one of the two possibilities in Fig.~\ref{f-3}. We will consider the  Ptolemy relation for each of these possibilities.

  Notice that $k \leq a_i$ since $k$ is the length of the shortest bridging arc on $S$ and all $\alpha_i$ are bridging arcs. Moreover, at most one of any two adjacent boundary arcs on $\Sigma$ can have length $k$ while the other is of length 1. And only one boundary component for $\Sigma$ contains arcs of length $k$. This implies that we are left to consider the five cases I-V listed in Fig. \ref{f-4}. We will denote by $\alpha_i'$ the diagonal crossing $\alpha_i$ in the quadrilateral $q_i$.

\begin{figure}[h]
\centering
\scalebox{0.8}[0.7]{\def\svgwidth{5.8in}
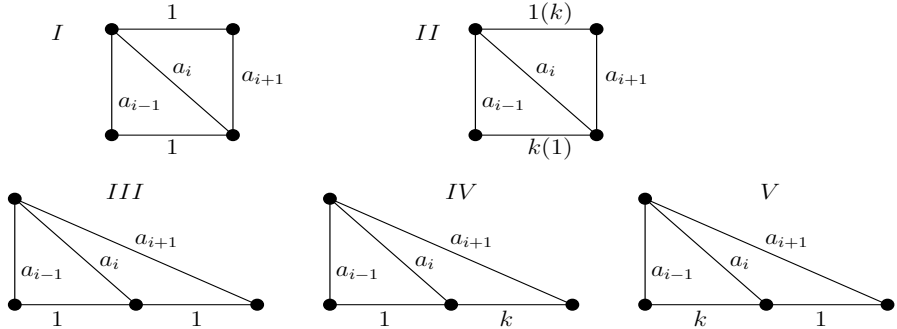}
\caption{Cases I - V.}
\label{f-4}
\end{figure}

  \begin{itemize}
  \item[(I)] By the  Ptolemy relation we have
    $$ a_ia_{i}'=a_{i-1}a_{i+1}+1.  $$
    Here, we know that $a_{i-1}\le a_i$ by inductive assumption and $a_i\le a_i'$ by the construction of the triangulation. Moreover, we also know $a_j\ge k> 1$ for any $j$.
    Suppose that $a_{i+1}<a_i$, i.e. $a_{i+1}=a_i-s$, for $s \geq 1$. Then
    $$a_ia_i'=a_{i-1}(a_i-s)+1= a_{i-1}a_i+(1-s\cdot a_{i-1})\le a_i a_i' + (1-s\cdot a_{i-1}) < a_i a_i',
    $$
    which is a contradiction.

  \item[(II)] In this case the Ptolemy relation gives    
   $$ a_ia_{i}'=a_{i-1}a_{i+1}+k,  $$
and similarly to the above we assume  $a_{i+1}=a_i-s$, $s\ge 1$, and compute  
    $$a_ia_i'=a_{i-1}(a_i-s)+k= a_{i-1}a_i+(k-s\cdot a_{i-1})\le a_i a_i' + (k-s\cdot a_{i-1}) \le a_i a_i'.
    $$
    Notice that the only way to avoid the contradiction is when both inequalities are identities, that is when
    $s=1$, $a_{i-1}=k$, and $a_{i-1}=a_i$ and $a_i=a_i'$, i.e. when
    $a_{i-1}=a_i=a_i'=k$. However, in that case $a_{i+1}=a_i-s=k-1<k$ which is impossible by the assumption that $k$ is the shortest $\lambda$-length of a bridging arc.

  \item[(III)] By the Ptolemy relation we have
    $$ a_ia_{i}'=a_{i-1}+ a_{i+1},  $$
     assuming  $a_{i+1}=a_i-s$, $s\ge 1$, we get  
     $$ a_ia_{i}'=a_{i-1}+ a_{i+1}=a_{i-1}+a_i-s \le 2a_i-s.  $$
     Notice that $\alpha_i'$ is not a bridging arc in this case (and was not bridging in $S$), so, we cannot assume $a_i'\ge a_i$.
     However, we know $a_i'\ge 2$ as the surface $S$ is reduced. So, we have $2a_i\le a_ia_i'\le 2a_i-s$ which is clearly impossible
     for  $s\ge 1$.

   \item[(IV)]   By the Ptolemy relation we have
    $$ a_ia_{i}'=ka_{i-1}+ a_{i+1}. $$
    Notice that in this case the arc $\alpha_i'$ is a lift of a bridging arc, and hence we have $a_i'\ge k$. 
     Assuming  $a_{i+1}=a_i-s$, $s\ge 1$, we get  
    $$ a_ia_{i}'=ka_{i-1}+ a_{i+1}=ka_{i-1}+a_i-s\le (k+1)a_i -s.  $$
    If $a_{i}'\ge k+1$, this is impossible. The case $a_i'=k$ is also impossible in view of  \emph{Claim 1}  (applied to $a_i'$ is in the role of $\beta$).

  \item[(V)]   By the Ptolemy relation we have
    $$ a_ia_{i}'=a_{i-1}+ ka_{i+1}.  $$
    Again, the arc $\alpha_i'$ is a lift of a bridging arc hence we have $a_i'\ge k$. 
     Assuming  $a_{i+1}=a_i-s$, $s\ge 1$, we get  
    $$ a_ia_{i}'=a_{i-1}+ ka_{i+1}=a_{i-1}+k(a_i-s)\le a_i + ka_i -ks=(k+1)a_i -ks,   $$
    which is impossible when $a_i'\ge k+1$. Since $a_i'\ge a_i\ge k$, this implies that $a_i'=a_i=k$, but then $a_{i+1}\le k-1$
    and this is impossible by construction.
  
\end{itemize}

We conclude that assuming $a_{i+1} < a_i$ in any of the five cases leads to a contradiction.
\\{\it End of the proof of Claim 2.}

\medskip
We see from the \emph{Claim 2} that the sequence of $\lambda$-lengths $(a_i)$ is monotone increasing. Since the arcs are lifts of a finite number of arcs from the annulus, this is only possible if all these arcs have the same  $\lambda$-length, denote it by $a$. Then we have $a\ge k$. 

Let $t_i$ and $t_{i+1}$ be two  adjacent triangles in the triangulation $T$ of $\Sigma$, and let $q_i=t_i\cup t_{i+1}$ be the quadrilateral composed of them. We will say that $q_i$ is a \emph{good quadrilateral} if the boundary arcs of both triangles
  $t_i$ and $t_{i+1}$ are of unit $\lambda$-length, and these boundary arcs lie on different components of the boundary of $\Sigma$ (as in Case I in Fig.~\ref{f-4}).

  \medskip
  \noindent
  {\bf Claim 3:} If $q_i$ is a good quadrilateral, then
  $a_{i+1}> a_i$.

\medskip
  \noindent
  {\it Proof of Claim 3:} We use the Ptolemy relation as in the proof of Case (I) of \emph{Claim 2}, but this time we assume $s=0$ (i.e. $a_{i+1}=a_i$).
 Then we get
 $ a_i a_i'=a_{i-1}a_i+1$, which is impossible since the left hand side is divisible by $a_i > 1$, and the right hand side is not.\\
 {\it End of proof of Claim 3.}  

\medskip



The following statement completes the proof of part (a).

   \medskip
  \noindent
  {\bf Claim 4:} $T$ contains at least one good quadrilateral $q_i$.

  \medskip
  \noindent
  {\it Proof of Claim 4:} Suppose that $T$ is a triangulation by bridging arcs containing no good quadrilaterals.
  
  Let $u_1,\dots, u_m$ be the marked points on one boundary of the annulus $S'$, and $v_1,\dots,v_n$ the points on the other boundary. Without loss of generality we assume that the boundary arcs $u_1u_2$ and $u_lu_{l+1}$,  are of $\lambda$-length $k$ (here $l\ge 3$ and $l+1\le m$), while all other boundary arcs are of $\lambda$-length 1.

\begin{figure}[h]
\centering
\scalebox{0.8}[0.8]{\def\svgwidth{5.2in}
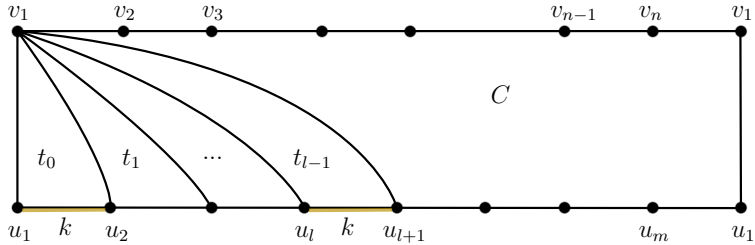}
\caption{Triangulation of $S'$.}
\label{f-gq}
\end{figure}

  Consider the triangle $t_1$ with boundary side $u_2u_3$ and suppose that $v_1$ is the third vertex of $t_1$ (we can assume this without loss of generality after shifting the numeration of the $v_i$'s). Since $T$ contains no good quadrilateral, $v_1u_1$ and $v_1u_4$ are arcs of $T$, otherwise, there is a good quadrilateral $v_1v_2u_3u_2$ or $v_1u_3u_2v_n$. Similarly, we see that all of  $v_1u_4,v_1u_5, \dots,v_1u_l$ should be arcs in $T$, and moreover, $v_1u_{l+1}\in T$. We obtain that all triangles  $t_0, \ldots, t_{l-1}$ lie in one fan with vertex $v_1$.

  Consider the $l$ triangles described above, and let $C$ be the complement of that region on $S'$: $C$ is a polygon with sides $v_1u_1$ and $v_1u_{l+1}$, and all other edges are boundary arcs of $\lambda$-length $1$ (see Fig. \ref{f-gq}). Notice that the restriction of $T$ to the polygon $C$ consists of bridging arcs, connecting the top boundary to the bottom. It is easy to see that such a triangulation contains a good quadrilateral since $l\ge 3$ and $l+1\le m$ so, as a minimum, the polygon $C$ is a square $v_1 u_{l+1} u_1 v_1$. Hence it has a bridging arc and it contains a good quadrilateral.\\
{\it End of proof of Claim 4 and part (a).}

\medskip

Now, we will prove (b). Suppose that $T$ and $T'$ are two unitary triangulations for the same frieze $\lambda(S,M)$.
Let $\alpha\in T'$ be an arc,  $\alpha\notin T$. Since $\alpha\notin T$, there is an arc $\gamma\in T$ crossed by $\alpha$. Consider the Ptolemy relation for the quadrilateral with diagonals $\alpha$ and $\gamma$: it writes as
$1\cdot 1=a\cdot c +b\cdot d$ where $a,b,c,d\in \mathbb Z_+$, which is impossible.
This implies that $T$ and $T'$ cannot be both unitary, and thus the unitary triangulation is unique.

Part (c) can be  easily shown by induction on the number of triangles in one fan crossed by $\alpha$ (by applying Ptolemy relation to prove each step of the induction). \end{proof}

\begin{remark} Note that for surfaces that have more than three boundary components or handles, there will be at least one additional case to those in {\bf Claim 2}. At this point, we are not able to treat those cases.
\end{remark}

{\bf Acknowledgements:} The authors would like to thank the Isaac Newton Institute for Mathematical Sciences, Cambridge, for support and hospitality during the programme ``Cluster algebras and representation theory'' where work on this paper was undertaken. This work was supported by EPSRC grant no. EP/R014604/1.
\printbibliography

\end{document}